\documentclass[11pt]{article}

\usepackage{amssymb, enumerate}
\usepackage{amsfonts}
\usepackage{amsmath}
\usepackage{color}
\usepackage{hyperref}

\newcommand{\R}{\mathbb R}
\newcommand{\C}{\mathbb{C}}
\newcommand{\Z}{\mathbb{Z}}
\newcommand{\N}{\mathbb{N}}
\newcommand{\Q}{\mathbb{Q}}

\newcommand{\cA}{\mathcal{A}}

\newtheorem{proposition}{Proposition}
\newtheorem{theorem}[proposition]{Theorem}
\newtheorem{definition}[proposition]{Definition}
\newtheorem{remark}[proposition]{Remark}

\newtheorem{lemma}[proposition]{Lemma}
\newtheorem{example}{Example}

\newenvironment{proof}{
\trivlist \item[\hskip \labelsep\mbox{\it Proof.
}]}{\hfill\mbox{$\square$}
\endtrivlist}

\title{Puiseux expansions and non-isolated points in algebraic varieties\footnote{Partially supported by the following Argentinian grants: PIP 0099/11 CONICET and UBACYT 20020120100133 (2013/2016).}}
\author{Mar\'\i a Isabel Herrero$^{\sharp, \diamond}$, Gabriela Jeronimo$^{\sharp,\dag,\diamond}$, Juan
Sabia$^{\dag,\diamond}$\\[5mm]
{\small $\sharp$ Departamento de Matem\'atica, Facultad de
Ciencias Exactas y
Naturales,} \\[-1mm] {\small Universidad de Buenos Aires, Ciudad
Universitaria, (1428) Buenos Aires, Argentina}\\[2mm]
{\small $\dag$ Departamento de Ciencias Exactas, Ciclo B\'asico
Com\'un,}\\[-1mm]
{\small Universidad de Buenos Aires, Ciudad Universitaria, (1428)
Buenos Aires, Argentina}\\[2mm]
{\small $\diamond$ IMAS, CONICET, Argentina }}

\begin{document}

\maketitle

\begin{abstract}
We consider the problem of deciding whether a common solution to a multivariate polynomial equation system is  isolated or not. We present conditions on a given truncated Puiseux series vector centered at the point ensuring that it is not isolated. In addition, in the case that the set of all common solutions of the system has dimension $1$, we obtain further conditions specifying to what extent the given vector of truncated Puiseux series coincides with the initial part of a parametrization of a curve of solutions passing through the point.
\end{abstract}

\bigskip


\section{Introduction}

A usual way to describe the set of complex common zeros
$V(\mathbf{f})$ of a finite family of multivariate polynomials $\mathbf{f}$ with
rational coefficients is by means of the equidimensional
decomposition of the algebraic variety $V(\mathbf{f})$. Several
general algorithmic symbolic procedures computing  polynomials characterizing each
equidimensional component have been proposed (see, for example,
\cite{GH91}, \cite{EM99}, \cite{JS02}, \cite{Lecerf03} and \cite{JKSS04}).

An alternative encoding of an equidimensional variety that
originated in the numerical algebraic geometry framework is by
means of a \textit{witness point set}, namely a suitable linear slicing of the variety consisting of a finite set of points containing as many points as the degree of the
variety (see \cite[Definition 13.3.1]{SW05}). This representation
has been applied for algorithmic numerical equidimensional and irreducible
decomposition (see \cite[Chapters 13-15]{SW05}). In this context, for instance, the software package PHCpack implements homotopy continuation methods to compute a numerical irreducible decomposition (see \cite{SVW03}). However, numerical approaches are subject to ill conditioning, which may lead to propagation of roundoff errors and inconclusive results.

In the symbolic framework, for certain families of polynomial systems, larger sets of points
representing the equidimensional components of a variety (called \textit{witness supersets}, as introduced in
\cite[Definition 13.6.1]{SW05}) can be computed with better
complexities than in previous symbolic decomposition procedures (see for instance,
\cite{HJS13}, where the case of sparse polynomial systems with $n$
equations in $n$ variables is considered), but no algorithm
discarding extra points within the same complexity order is known.
This motivates the search for new symbolic tools that may lead to
solve this problem.

A first question that arises in this context is, given a point $ \xi
\in V(\mathbf{f})$, to decide algorithmically  whether it is
isolated or not (numerical algorithms dealing with this task can be
found in \cite[Section 13.7.2]{SW05}, \cite{BHPS09} or \cite{KL08}).

For a system of two polynomials in two variables, in
\cite[Proposition 5.3]{AV11} it is stated that, under certain
hypotheses, if the second term in the Puiseux series expansion at a
common root $\xi$ can be computed, then there exists a curve of
solutions for the original system; however, the result does not hold
for arbitrary bivariate polynomial systems. In \cite{AV12} the authors
extend this result to the case of $n$ variables and apply it
successfully to produce exact representations for solution sets of the cyclic $n$-roots problem.

In this paper we give conditions a vector of truncated Puiseux
series  $\Theta$ centered at a point $\xi\in V(\mathbf{f})$ must
fulfill in order to ensure that $\xi$ is not an isolated point of
$V(\mathbf{f})$ in the general case. Moreover, if the dimension of
$V(\mathbf{f})$ is $1$, we give further conditions on $\Theta$ to
ensure that its initial part  coincides with the initial part of a
Puiseux series expansion of a parametrization of a curve in
$V(\mathbf{f})$ containing $\xi$. We will assume that $\mathbf{f}$ does not vanish identically at $\Theta$ since, otherwise, the results follow straightforwardly.

We first consider the case of two polynomials in two variables by
means of elementary resultant-based techniques. The given conditions
depend on the degrees of the polynomials involved.
Then,
we deal with the general case of $n$-variate polynomial systems,
obtaining conditions that depend on invariants associated to the
ideal the polynomials generate and the degree of the variety they
define.

\section{Bivariate polynomials}\label{sec:2var}

The branches of a plane curve in $\C^2$  through a point  can be
locally parametrized by means of Puiseux series (see, for example,
\cite{Walker}). A Puiseux series with complex coefficients
centered at $\xi_1 \in \C$ is a formal expression of the form
$\sum\limits_{i\in \N_0} a_i (t-\xi_1)^{\gamma_i}$, where $a_i\in
\C$ for every $i\in \N_0$, and $\{\gamma_i\}_{i \in\N_0}$ is a
family of rational numbers with bounded denominators such that
$\gamma_i<\gamma_{i+1}$ for every $ i \in \N_0$. We denote
$\textrm{ord}_{(t-\xi_1)} (\sum\limits_{i\in \N_0} a_i
(t-\xi_1)^{\gamma_i})= \min\{ \gamma_i \mid a_i\ne 0\}$ the order
of this Puiseux series. We will write $\C\{\{t-\xi_1\}\}$ for the
ring of all Puiseux series with complex coefficients centered at
$\xi_1$.

Given a polynomial $q\in \C[t,Y]$ and a point $\xi=(\xi_1,\xi_2) \in \C^2$ such that $q(\xi)=0$ and $q(\xi_1, Y) \not\equiv 0$, there is at least one  Puiseux series $\sum\limits_{i\in \N_0}a_i(t-\xi_1)^{\gamma_i}$ such that $\gamma_0=0$, $a_0 = \xi_2$ and  $q(t, \sum\limits_{i\in \N_0}a_i(t-\xi_1)^{\gamma_i})=0$ (such a Puiseux series will be called a parametrization of the curve through $\xi$).
In the following lemma we give conditions on the vanishing order of the polynomial $q$ at a truncated Puiseux series to establish to what extent it coincides with a parametrization of the curve defined by $q$.

\begin{lemma}\label{lem:aprox} Let $q\in \C[t, Y]$ and $\xi = (\xi_1, \xi_2)\in \C^2$ such that $q(\xi) =0$. Let $\theta=\sum\limits_{i=0}^N a_i(t-\xi_1)^{\gamma_i}$ with $a_0=\xi_2$ and
$\gamma_0, \dots, \gamma_N, L \in \Q$ such that
$0=\gamma_0<\dots<\gamma_N\le L$ and ${\rm
ord}_{(t-\xi_1)}(q(t,\theta))>L$.
If $L\ge {\rm mult}(\xi_1, c)$, where $c\in \C[t]$ is the
leading coefficient of $q$ as a polynomial in $\C[t][Y]$, then
there is a parametrization of a branch of the curve $V(q)\subset
\C^2$ through $\xi$ whose initial terms are $(t,\sum\limits_{i=0}^M
a_i(t-\xi_1)^{\gamma_i})$, where $M= \max \{ i \in \{0,\dots, N\}
\mid \gamma_i \le \dfrac{L-{\rm mult}(\xi_1, c)}{\deg_Y(q)}\}$.
\end{lemma}

\begin{proof}
Consider $q$ as a polynomial in $ \C\{\!\{t-\xi_1\}\!\}[Y]$ and its linear factorization
$q=c \prod\limits_{h=1}^{D}(Y-\eta_h)$, where $c\in \C[t]$. If no $\eta_h\in \C\{\!\{t-\xi_1\}\!\}$ begins with
$\sum\limits_{i=0}^{M}a_i(t-\xi_1)^{\gamma_i}$, by the definition
of $M$ we have that $\textrm{ord}_{(t-\xi_1)}q(t,\theta)= {\rm mult}(\xi_1, c) + \sum\limits_{h=1}^D \textrm{ord}_{(t-\xi_1)}(\theta - \eta_h)\le  {\rm mult}(\xi_1, c) + D \, \dfrac{L-{\rm mult}(\xi_1, c)}{D} = L$,
contradicting the assumption that $\textrm{ord}_{(t-\xi_1)}(q(t,\theta))>L$.
\end{proof}

The following example shows that the bound given in the previous lemma can be attained:

\begin{example}  Let $q \in \C[t, Y]$ be the polynomial $q(t,Y) = t^{d_1}(Y-1)^{d_2}$ and $\xi=(0,1)$ a zero of $q$. Here, $(t,1)$ is a curve in $V(q)$ passing through $\xi$.  Let $\gamma, L \in \Q$ such that $L\ge d_1$ and $0< \gamma\le L$. Then $\theta= 1+t^\gamma$ satisfies that ${\rm{ord}}_t(q(t,\theta)) >L$  if and only if
$\dfrac{L-d_1}{d_2}<\gamma$, and the bound given by the lemma in this case is exactly
$\dfrac{L-d_1}{d_2}$.
\end{example}

\begin{remark}\label{rem:NH}
If  $q(\xi) = 0$ and $\cfrac{\partial
q}{\partial X_2}(\xi)\ne 0$, there exists a unique
formal power series with integer exponents $\sum\limits_{i \in \N_0} c_i(t-\xi_1)^i$ such that $q(t,\sum\limits_{i \in \N_0} c_i(t-\xi_1)^i)=0$ and $c_0=\xi_2$, and the Newton-Hensel lifting
gives a constructive way to approximate it
(see \cite[Lemma 3]{HKPSW00}, \cite{GLS01} for algorithmic
versions of this result). In this case, under the assumptions of
Lemma \ref{lem:aprox}, $a_i = c_i$ for all $i\le L-{\rm mult}(\xi_1, c)$. This can be proved following the arguments in the proof of the lemma and using that there is at most one root
$\eta_{h}$ of $q$ such that ${\rm{ord}}_{(t-\xi_1)}(\xi_2-\eta_h)>0$.
\end{remark}

Now we analyze our main problem in the case of  two bivariate polynomials. Consider first the following easy example:

\begin{example}
Let $f_1, f_2\in \C[X_1,X_2]$ be the polynomials
$$f_1(X_1,X_2)=X_2-1+X_1+X_1^{d_1}\mbox{ and }f_2(X_1,X_2)=
X_2-1+X_1+X_1^{d_2}.$$ It is clear that the zero sets of $f_1$ and $f_2$ in $\C^2$ are the curves parametrized by $(t, 1 - t - t^{d_1})$  and  $(t, 1 - t - t^{d_2})$ respectively and, so, if $d_1 \ne d_2$, there is no curve of common
zeroes for these polynomials. Nonetheless, the terms of degree lower than $\min\{d_1,d_2\}$ of both expansions coincide. \end{example}

The question that arises is, given two polynomials $f_1,f_2 \in \C[X_1,X_2]$ with a common solution $\xi$, to what extent the Puiseux series expansions of parametrizations of curves of solutions through $\xi$ of $f_1$ and $f_2$ respectively must coincide in order to be able to conclude that $f_1$ and $f_2$ share a curve of solutions.

In  \cite[Proposition 5.3]{AV11}, it is stated that, given
$f_i(X_1,X_2)=p_i(X_2)+P_i(X_1,X_2)$ for $i=1,2$, where $p_i$ have nonzero
constant term and all terms in $P_i$ have a positive power in $X_1$, and $\xi_2 \in \C-\{0 \}$ such that
$p_i(\xi_2)=0$, $\cfrac{\partial p_i}{\partial X_2}(\xi_2)\ne 0$
and $f_i(t,\xi_2)\neq 0$  for $i=1,2$, if the
exponents and coefficients of the first two terms $(X_1 = t, X_2 =
\xi_2 + a_1t^{\gamma_1})$ of the series expansions at the
common root $\xi= (0,\xi_2)$ coincide, there exists a curve of common
solutions containing $\xi$ and these first two terms are in fact the leading part of a
Puiseux series expansion of a regular common factor of $f_1$ and
$f_2$. As we can see in the previous example, this is not always
the case.

The example also shows that the  precision required to ensure the existence of a curve of common zeros
containing $\xi$ depends on the degrees of the polynomials involved. Here we present a lower bound
for this precision.

\begin{proposition} \label{lem:bivariate} Let $f_1, f_2 \in \C[X_1, X_2]$
be polynomials with positive degree in the variable $X_2$ and with a common zero $\xi=(\xi_1,\xi_2)$. Let
$\theta=\sum\limits_{i= 0}^Na_i(t-\xi_1)^{\gamma_i}$ with
$a_0=\xi_2$ and $\gamma_0, \dots, \gamma_N, L \in \Q$ such that
$0=\gamma_0<\dots<\gamma_N\le L$ and ${\rm
ord}_{(t-\xi_1)}(f_j(t,\theta))>L $ for $j=1,2$. Let
$d_{ij}=\deg_{X_i}(f_j)$. If $L\ge d_{11}d_{22}+d_{12}d_{21},$
then there exists a curve of common zeroes of $f_1$ and $f_2$ that
contains $\xi$. Moreover, there is a parametrization of the curve
whose initial terms are $(t,\sum\limits_{i=
0}^{M}a_i(t-\xi_1)^{\gamma_i})$, where
$$M= \max\{i \in \{0, \dots, N\} \ | \ \gamma_i\le
\frac{L-(d_{11}d_{22}+d_{12}d_{21})-{\rm min}\{d_{11},
d_{12}\}}{{\rm min}\{d_{21}, d_{22}\}}+d_{11}+d_{12}\}.$$
\end{proposition}

\begin{proof}
If the resultant $\mbox{Res}_{X_2}(f_1,f_2)$ is not the zero
polynomial, then $\deg_{X_1}(\mbox{Res}_{X_2}(f_1,f_2)) \le
d_{11}d_{22} +d_{12} d_{21} $ and therefore,
$\textrm{ord}_{(t-\xi_1)}(\mbox{Res}_{X_2}(f_1,f_2)(t))\le L$. On
the other hand, since the order of any linear combination of $f_1$
and $f_2$ with coefficients in $\C[X_1,X_2]$, evaluated in
$(t,\theta)$ is higher than $L$, it follows that
$\textrm{ord}_{(t-\xi_1)}(\mbox{Res}_{X_2}(f_1,f_2)(t))>L$.
Therefore, if $L\ge d_{11}d_{22} +d_{12} d_{21}$, we have that
$\mbox{Res}_{X_2}(f_1,f_2)= 0$ and so, $f_1$ and $f_2$ have a
common factor depending on $X_2$.

Let $q:=\mbox{gcd}(f_1,f_2)\in\C[X_1,X_2]$ and let
$\widetilde{f}_1$, $\widetilde{f}_2$ be such that
$f_1=q\widetilde{f}_1$ and $f_2=q\widetilde{f}_2$.  If $q(\xi) \ne
0$, then  $\textrm{ord}_{(t-\xi_1)}(\widetilde f_j(t,\theta)) =
\textrm{ord}_{(t-\xi_1)}(f_j(t,\theta)) >L$ and, repeating the
arguments above, $\mbox{Res}_{X_2}(\widetilde f_1,\widetilde
f_2)=0$, which is a contradiction.

Similarly, as
$\mbox{Res}_{X_2}(\widetilde{f}_1,\widetilde{f}_2)(t)\neq 0$,  it
follows that, for some $j$, $\textrm{ord}_{(t-\xi_1)}(\widetilde f_j(t,\theta)) \le
(d_{11}-d_{1q})(d_{22}-d_{2q})+(d_{12}-d_{1q})(d_{21}-d_{2q})$,
where $d_{iq}=\deg_{X_i}(q)$ for $i=1,2$. Then,
\[\begin{split}
 \textrm{ord}_{(t-\xi_1)} q(t,\theta)  & = \textrm{ord}_{(t-\xi_1)} f_j(t,\theta) - \textrm{ord}_{(t-\xi_1)} \widetilde{f}_j(t,\theta) \\
 & > L - (d_{11}-d_{1q})(d_{22}-d_{2q})-(d_{12}-d_{1q})(d_{21}-d_{2q})\\
 \end{split}
 \]
Since  $L - (d_{11}-d_{1q})(d_{22}-d_{2q})-(d_{12}-d_{1q})(d_{21}-d_{2q}) = L - (d_{11} d_{22} +d_{12} d_{21} )+ d_{2q} (d_{11}+d_{12}-2d_{1q}) + d_{1q}(d_{21}+d_{22})\ge d_{1q}$, by Lemma \ref{lem:aprox}, there is a parametrization of a curve passing through $\xi$ and contained in $V(q) \subset V(f_1, f_2)$ whose initial terms are $a_i(t-\xi_1)^{\gamma_i}$ as long as
\[ \gamma_i \le \dfrac{L - (d_{11}-d_{1q})(d_{22}-d_{2q})-(d_{12}-d_{1q})(d_{21}-d_{2q})- d_{1q}}{d_{2q}}.\]

As
\[ \begin{split}
& \dfrac{L - (d_{11}-d_{1q})(d_{22}-d_{2q})-(d_{12}-d_{1q})(d_{21}-d_{2q})- d_{1q}}{d_{2q}} \\
& =
\dfrac{ L - (d_{11} d_{22} +d_{12} d_{21} )+ d_{1q}(d_{21}+d_{22}-2d_{2q})- d_{1q}}{d_{2q}}+ d_{11}+d_{12}\\
& \ge \dfrac{L-(d_{11}d_{22}+d_{12}d_{21})-\min\{d_{11},
d_{12}\}}{{\rm min}\{d_{21}, d_{22}\}}+d_{11}+d_{12},
\end{split}\]
the proposition follows.
\end{proof}

As before, the given bound can be attained:

\begin{example}  Let $f_1, f_2 \in \C[X_1, X_2]$ be the polynomials
$$f_1(X_1,X_2)= X_1^{d_{11}}(X_2-1)^{d_2}, \ f_2(X_1,X_2)=X_1^{d_{12}}(X_2-1)^{d_2}$$ and $\xi=(0,1)$ a
common zero. Here, $(t,1)$ is a curve of common solutions of $f_1$
and $f_2$ containing $\xi$. Let $\gamma, L \in \Q$ such that $L\ge
(d_{11}+d_{12})d_2$ and $ \gamma\le L$. The vector $(t,
1+t^\gamma)$ satisfies the hypothesis of Proposition
\ref{lem:bivariate} if and only if
$\dfrac{L-\min\{d_{11},d_{12}\}}{d_2}<\gamma$, which is exactly
the same bound given by the proposition.
\end{example}

\begin{remark}
If for $j=1$ or $j=2$, we have that $f_j(\xi) = 0$ and $\cfrac{\partial
f_{j}}{\partial X_2}(\xi)\ne 0$, under the assumptions  of
Proposition \ref{lem:bivariate}, following Remark \ref{rem:NH}, there is a common solution curve for $f_1$ and $f_2$ having a parametrization whose initial
terms are $(t,\sum\limits_{i= 0}^{K}a_i(t-\xi_1)^i))$ for
$K=\lfloor L\rfloor-\min\{d_{11},d_{12}\}$.
\end{remark}

\section{Arbitrary Systems}\label{sec:nvar}

In this section we are going to extend the results of Section
\ref{sec:2var} to arbitrary multivariate polynomial equation
systems.

\subsection{Non-isolated points}

Let $\mathbf{f}=(f_1,\dots, f_m)$ be a polynomial system in
$\C[X_1,\dots,X_n]$ and $V(\mathbf{f})=V(f_1,\dots, f_m)$ be the
set of the common zeros of $\mathbf{f}$ in $\C^n$. Let
$\xi=(\xi_1,\dots, \xi_n)\in \C^n$ be a point in $V(\mathbf{f})$.
The next theorem presents a bound for the vanishing order required
on the system evaluated at a vector of truncated Puiseux series
centered at $\xi$ to ensure that $\xi$ lies in an irreducible
component $W$ of $V(\mathbf{f})$ such that
$\overline{\pi_{X_1}(W)} = \C$, where $\pi_{X_1}:\C^n\rightarrow
\C$ is the projection to the first coordinate, $\pi_{X_1}(x_1,
\dots, x_n)=x_1$, and the closure is taken with respect to the
Zariski topology. The bound is given in terms of the \emph{Noether
exponent} of the ideal $\langle f_1,\dots, f_m\rangle$, that is,
the minimum positive integer $e(\mathbf{f})$ such that
$\big(\sqrt{\langle f_1,\dots,
f_m\rangle}\big)^{e(\mathbf{f})}\subset \langle f_1,\dots,
f_m\rangle$.

In the sequel, for an irreducible variety $C\subset \C^n$ such
that $\overline{\pi_{X_1}(C)} = \C$, we will say that $C$ is a
variety \emph{with free variable $X_1$}.

\begin{theorem} \label{theo:curve} Let $\mathbf{f}=(f_1,\dots, f_m)$ be a polynomial
system in $\C[X_1, \dots, X_n]$ and $\xi=(\xi_1,\dots, \xi_n)\in \C^n$ be a
zero of $\mathbf{f}$. Let $\gamma_0, \dots, \gamma_N, L \in \Q$
such that $0=\gamma_0<\dots<\gamma_N\le L$ and
$$\Theta=(t,
\sum\limits_{i=0}^Na_{i2}(t-\xi_1)^{\gamma_i},
\dots,\sum\limits_{i=0}^Na_{in}(t-\xi_1)^{\gamma_i})$$ be a
Puiseux series vector with coefficients in $\C$ centered at $\xi_1$
such that $a_{0l}=\xi_l$ for all $2 \le l \le n$ and ${\rm
ord}_{(t-\xi_1)}(f_j(\Theta))>L$ for all $ 1 \le j \le m$. Let
$e(\mathbf{f})$ be the Noether exponent of $\langle f_1,\dots,
f_m\rangle$. If $L\ge e(\mathbf{f}),$ then there exists an
irreducible component $W$ of $V(\mathbf{f})$ with free variable
$X_1$ such that $\xi \in W$.
\end{theorem}

\begin{proof}
Let $\mathcal{V}$ be the algebraic variety of all irreducible
components of $V(\mathbf{f})$ with free variable $X_1$. If $\mathcal{V} = V(\mathbf{f})$, there is nothing to prove.

Let $p\in \C[X_1]$ be the monic polynomial of minimum degree that
vanishes over  $\pi_{X_1}(V(\mathbf{f})- \mathcal{V})$. Suppose
$\mathcal{V}=\emptyset$; then $p\in \sqrt{ \langle f_1, \dots,
f_m\rangle}$ and so, $p^{e(\mathbf{f})}\in \langle f_1, \dots,
f_m\rangle$.  Hence,
 ${\rm ord}_{(t-\xi_1)}(p(\Theta)^{e(\mathbf{f})})=
{e(\mathbf{f})}{\rm ord}_{(t-\xi_1)}(p(t))\le e(\mathbf{f})\le L$.
As the order of any linear combination of $f_1, \dots, f_m$
evaluated in $\Theta$ is higher than $L$, it follows that
$\mathcal{V}\neq\emptyset$.

Assume now that $\xi \not\in \mathcal{V}.$ Then, for every
irreducible component $C$ of $V(\mathbf{f})$ such that $\xi \in
C$, $\pi_{X_1}(C)=\xi_1$. Let $q \in \C[X_1, \dots, X_n]$ such
that $q$ vanishes over the union of all irreducible components of
$V(\mathbf{f})$  that do not contain $\xi$ and $q(\xi)\neq 0$.
Then $(X_1-\xi_1)q \in \sqrt{\langle f_1, \dots, f_m\rangle}$ and
so, $((X_1-\xi_1)q)^{e(\mathbf{f})}$ is a linear combination of
$f_1, \dots, f_m$ with coefficients in $\C[X_1, \dots, X_n]$;
then, ${\rm ord}_{(t-\xi_1)} (((t-\xi_1)
q(\Theta))^{e(\mathbf{f})}) > L$.  But ${\rm
ord}_{(t-\xi_1)}q(\Theta) = 0$. This leads to a contradiction and,
consequently, $\xi\in \mathcal{V}$.
\end{proof}

The following trivial example shows that the bound $L\ge e(\mathbf{f})$ in the previous theorem is sharp.

\begin{example} Let $\mathbf{f}=(x_1^e, x_2,\dots, x_n)$. Then, $V(\mathbf{f}) = \{ 0\}$ and it is easy to see that $e(\mathbf{f}) = e$. Consider $\xi= 0$ and $\Theta =(t,0,0,\dots, 0)$.  Then ${\rm{ord}}_t(f_1(\Theta)) = e$ and $f_j(\Theta)=0$ for every $2\le j \le n$. However, $V(\mathbf{f})$ contains no curve.
\end{example}

\begin{remark} \label{rem:noether} Any explicit upper bound for $e(\mathbf{f})$ provides an explicit bound for the parameter $L$  in Theorem \ref{theo:curve}. For instance, the following bounds could be applied:

\begin{itemize}
\item If $\deg(f_j)\le d$ for every $1\le j \le m$, then $e(\mathbf{f})
\le d^{\min\{n,m\}}$ (see \cite[Theorem 1.3]{Jelonek2005}).

\item For polynomials $f_1,\dots, f_m$ with supports
$\cA_1,\dots, \cA_m\subset (\Z_{\ge 0})^n$ respectively ($\cA_i$
is the set of vectors of exponents of the monomials of $f_i$ with
nonzero coefficients for all $1 \le i \le m$),  $e(\mathbf{f}) \le n^{n+2}
n! {\rm vol}_n(\cA \cup\Delta_n)$, where  $\cA =
\bigcup_{j=1}^m \cA_j$ and $\Delta_n$ is the standard simplex of
$\R^n$ (see \cite{Sombra99}).
Under certain assumption on an associated polytope, the following smaller bound holds:
$e(\mathbf{f}) \le \min\{n+1, m\}^2 n! {\rm vol}_n(\cA \cup\Delta_n)$
(see \cite[Theorem 2.10]{Sombra99}).
\end{itemize}
\end{remark}

\subsection{Varieties of dimension 1}\label{sec:dimV le1}

Under certain assumptions,  for a point $\xi$ in an algebraic variety $V(\mathbf{f})$, Theorem \ref{theo:curve} in the previous section ensures the
existence of a positive dimensional component of $V(\mathbf{f})$ with free variable $X_1$ containing
$\xi$.

If, in addition to the conditions
of Theorem \ref{theo:curve}, $\dim(V(\mathbf{f}))= 1$, a question that arises naturally is
to what extent the given Puiseux series vector coincides with the
expansion of a parametrization of a curve in $V(\mathbf{f})$
containing the point $\xi$. In this section we give a degree bound
that enables us to answer this question. The bound depends on the Noether exponent of the ideal and the degree $\deg(V(\mathbf{f}))$ of the variety (for the definition of degree we use, see \cite{Heintz83}).

In order to deal with
$1$-dimensional varieties, we will use the notion of a
\emph{geometric resolution}, widely used in computational
algebraic geometry (see for instance \cite{GLS01}).

\begin{definition}
 Let $V= \{\xi^{(1)}, \dots, \xi^{(D)}\}\subset \overline k^n$ be a
zero-dimensional variety defined by polynomials in $k[X_1,\dots, X_n]$, where $k$ is a field of characteristic $0$ and $\overline k$ an algebraic closure of $k$. Given a
linear form $\ell = \ell_1X_1 + \dots + \ell_n X_n$ in $k[X_1,
\dots, X_n]$ such that $\ell(\xi^{(i)})\ne \ell(\xi^{(j)})$ if $i
\ne j$, the following polynomials completely characterize $V$:
\begin{itemize}
\item the minimal polynomial $q = \prod\limits_{i=1}^D(Y - \ell(\xi^{(i)})) \in k[Y]$ of $\ell$ over
the variety $V$ (where $Y$ is a new variable),
\item polynomials $v_1, \dots, v_n \in k[Y]$ with $\deg(v_j) < D$ for every $1 \le j \le n$
satisfying $\xi^{(i)} = ( v_1(\ell(\xi^{(i)})),\dots, v_n(\ell(\xi^{(i)})))$ for every $1\le i\le D$.
\end{itemize}

The family of univariate polynomials $(q, v_1,\dots, v_n) \in
k[Y]^{n+1}$ is called the \emph{geometric resolution} of $V$ (or the geometric resolution of $k[V]$) associated
with  $\ell$. We have
$$V = \{(v_1 (y), \dots, v_n (y)) \in \overline k^n \mid y
\in \overline k, q(y)=0\}.$$
\end{definition}

The notion of geometric resolution can be extended to any equidimensional variety. In our situation, it can be defined as follows:
Let $\mathcal{V}\subset \C^n$ be an equidimensional variety of dimension $1$ defined by polynomials in $\C[X_1,\dots, X_n]$ such that $X_1$ is free for each irreducible component of $\mathcal{V}$. By considering $\C(X_1) \otimes \C[\mathcal{V}]$, we are in a zero-dimensional situation, and  a \emph{geometric resolution of $\mathcal{V}$ with free variable $X_1$} is a geometric resolution $(q, v_{2},\dots, v_n) \in \C(X_1)[Y]^{n}$ of $\C(X_1) \otimes \C[\mathcal{V}]$ associated to a linear form $\ell \in \C[X_{2},\dots, X_n]$.

\begin{theorem} \label{theo:dim1} Let $\mathbf{f}=(f_1,\dots, f_m)$ be a polynomial
system in $\C[X_1, \dots, X_n]$ such that $\dim(V(\mathbf{f}))\le
1$ and $\xi=(\xi_1,\dots, \xi_n)\in \C^n$ be a zero of $\mathbf{f}$. Let
$\gamma_0, \dots, \gamma_N, L \in \Q$ such that
$0=\gamma_0<\dots<\gamma_N\le L$ and
$$\Theta=(t,\sum\limits_{i=0}^Na_{i2}(t-\xi_1)^{\gamma_i},
\dots,\sum\limits_{i=0}^Na_{in}(t-\xi_1)^{\gamma_i})$$ be a
Puiseux series vector with coefficients in $\C$ centered at
$\xi_1$ such that $a_{0l}=\xi_l$ for all $2 \le l \le n$ and ${\rm
ord}_{(t-\xi_1)}(f_j(\Theta))>L$ for all $ 1 \le j \le m$. Let
$e(\mathbf{f})$ be the Noether exponent of $\langle f_1, \dots,
f_m\rangle$. If $L\ge  e(\mathbf{f})\deg(V(\mathbf{f}))$, there exists a curve $W$ in
$V(\mathbf{f})$ with free variable $X_1$ such that $\xi \in W$ and
there is a parametrization of $W$ whose initial terms are
$\Theta_M:=(t,\sum\limits_{i=0}^M a_{i2}(t-\xi_1)^{\gamma_i},
\dots,\sum\limits_{i=0}^M a_{in}(t-\xi_1)^{\gamma_i})$, where
$$M=  \max \Big\{ i\in \{0,\dots, N\} \mid \gamma_i \le \dfrac{L}{e(\mathbf{f})
\deg(V(\mathbf{f}))}\Big\}.$$
\end{theorem}

\begin{proof}
By Theorem \ref{theo:curve}, the point $\xi$ lies in an
irreducible component $W$ of $V(\mathbf{f})$  with free variable
$X_1$. Let $\mathcal{V}$ be the union of all the irreducible
components of $V(\mathbf{f})$ with free variable $X_1$, which is a
nonempty equidimensional variety of dimension $1$. Replace $X_1$
by $t$ in the polynomials $f_1,\dots, f_m$, and consider the ideal
$\C(t) \otimes \langle \mathbf{f}\rangle \subset \C(t)[X_2,\dots,
X_n]$ and its zeros $\eta_1,\dots, \eta_D\in
\C\{\!\{t-\xi_1\}\!\}^{n-1}$. Let $\ell=\sum\limits_{k=2}^n\ell_k
X_k$ be a generic linear form and $(q, v_2, \dots, v_n)\in
\C(t)[Y]$ be the geometric resolution of $\mathcal{V}$ associated
with $\ell$. Then, $q(Y)=\prod_{h=1}^{D}\limits(Y-\ell(\eta_h))$ and $D\le \deg(\mathcal{V})$.
As $\ell$ is
generic, we may assume that, for every $1\le h\le D$,
$\textrm{ord}_{(t-\xi_1)}(\ell(\Theta)-\ell(\eta_h))=\min \{
\textrm{ord}_{(t-\xi_1)}
((\Theta)_k-(\eta_h)_k): 2\le k \le n\}$.

Let $\varphi: \mathcal{V} \to \C^2$ be the map $\varphi(x) =
(x_1,\ell(x_2,\dots, x_n))$. Then, there is a polynomial $c\in
\C[t]$ such that $\hat q(t,Y):=c(t) q(Y) \in \C[t,Y]$ and $\hat q$ defines
the Zariski closure of $\varphi(\mathcal{V})$. Note that
$\deg_t(c) + D \le \deg (\hat q )\le
\deg(\mathcal{V})$ (see \cite[Lemma 2]{Heintz83}).

As in the proof of Theorem \ref{theo:curve}, let $p\in \C[X_1]$ be
a monic polynomial of minimum degree that vanishes over
$\pi_{X_1}(V(\mathbf{f})- \mathcal{V})$ (if $\mathcal{V} = V(\mathbf{f})$ take $p=1$).  We have that $\deg(p)\le
\deg(V(\mathbf{f}))-\deg(\mathcal{V})$. Since $\hat q(X_1,
\ell)p(X_1)$ vanishes over $V(\mathbf{f})$, it follows that $(\hat
q(X_1, \ell)p(X_1))^{e(\mathbf{f})}$ is a linear combination of
$f_1, \dots, f_m$ with coefficients in $\C[X_1, \dots, X_n]$;
therefore, $\mbox{ord}_{(t-\xi_1)}((\hat q(t, \ell(\Theta))
p(t))^{e(\mathbf{f})})>L$. Then,
\[
\begin{split}
   \mbox{ord}_{(t-\xi_1)}(\hat q(t, \ell(\Theta)))
&  > \frac{L}{e(\mathbf{f})}-\deg(p)
   \ge \frac{L}{e(\mathbf{f})}- \deg(V(\mathbf{f}))+\deg(\mathcal{V}) \\
& \ge \Big(\frac{L}{e(\mathbf{f}) \deg(V(\mathbf{f}))}-1\Big)\deg(V(\mathbf{f}))+ D +\deg_t(c)\\
& \ge\frac{LD}{e(\mathbf{f}) \deg(V(\mathbf{f}))} + \deg_t(c).
\end{split}
\]

Since $c$ is the leading coefficient of $\hat q$ and $\deg_t(c) \ge {\rm mult}(\xi_1, c)$, by Lemma \ref{lem:aprox},  there exists $1\le h \le D$ such that
\[ \textrm{ord}_{(t-\xi_1)}( \ell(\Theta) - \ell(\eta_h)) > \dfrac{\frac{LD}{e(\mathbf{f}) \deg(V(\mathbf{f}))} + \deg_t(c) - {\rm mult}(\xi_1, c)}{D}\ge \frac{L}{e(\mathbf{f}) \deg(V(\mathbf{f}))}.\]
The theorem follows by our assumption on $\ell$.
\end{proof}

Although Theorem \ref{theo:dim1} states that, in the general case, a large number of terms only provide a few ones of the initial part of a parametrization, the following example shows that the precision order is sharp for certain choices of the parameters.

\begin{example} Let $\mathbf{f} = (\prod\limits_{k=1}^d (x_1 - kx_2)^e, x_3,\dots, x_n)$,
and $\xi = 0$ a common zero. It is easy to see that $e(\mathbf{f})=
e$ and $\deg(V(\mathbf{f})) = d$. Taking $L= ed$, the vector
$\Theta = (t, t+t^{1+\varepsilon},0,\dots,0)$ satisfies the
hypothesis from Theorem \ref{theo:dim1} for all $\varepsilon>0$
since ${\rm{ord}}_t(f_1(\Theta))= e (d+\varepsilon)>L$ and
$f_j(\Theta)=0$ for all $ 2 \le j \le n-1$. In this
case the precision bound from the previous theorem is exactly $\frac{L}{de}=1$ which
coincides with the first terms from $\Theta$ that correspond to a parametrization of a
curve in $V(\mathbf{f})$ for all $\varepsilon >0$.
\end{example}

\begin{remark}  Explicit upper bounds for both $e(\mathbf{f})$ and
$\deg(V(\mathbf{f}))$ provide explicit bounds for the parameters in Theorem \ref{theo:dim1}. For instance, using the bounds for $e(\mathbf{f})$ already stated in Remark \ref{rem:noether}, we have:

\begin{itemize}
\item If $\deg(f_j)\le d$ for every $1\le j \le m$, then
$\deg(V(\mathbf{f})) \le d^{\min\{n,m\}}$ (see \cite[Theorem
1]{Heintz83}). Therefore, for
$L\ge d^{2\min\{n,m\}}$ and $\Theta$ satisfying the conditions of the statement, if $\gamma_M \le L d^{-2\min\{n,m\}}$, then $\Theta_M$ is the initial part of a parametrization of a curve in $V(\mathbf{f})$ containing $\xi$.

\item For polynomials $f_1,\dots, f_m$ with supports
$\cA_1,\dots, \cA_m\subset (\Z_{\ge 0})^n$ respectively,
$\deg(V(\mathbf{f}))\le n! {\rm vol}_n(\cA \cup\Delta_n)$ (see
\cite[Proposition 2.12]{KPS01}), where $\cA = \bigcup_{j=1}^m \cA_j$ and $\Delta_n$ is the standard
simplex of $\R^n$. Therefore, using the bound for $e(\mathbf{f})$ stated in Remark \ref{rem:noether}, for $L \ge n^{n+2} (n! {\rm vol}_n(\cA \cup\Delta_n))^2$ and $\Theta$ satisfying the conditions of the statement, if $\gamma_M \le L\, n^{-n-2}(n! {\rm vol}_n(\cA \cup\Delta_n))^{-2}$, then $\Theta_M$ is the initial part of a parametrization of a curve in $V(\mathbf{f})$ containing $\xi$.

When $m=n$, the sharper bound
$\deg(V(\mathbf{f}))\le\mathcal{MV}_n(\cA_1\cup\Delta_n, \dots,
\cA_n\cup\Delta_n)$ holds (see \cite[Theorem 16]{HJS13}), where $\mathcal{MV}_n$ denotes the $n$-dimensional mixed volume. This leads to sharper bounds for $L$ and the precision order $\gamma_M$.
Under certain assumptions on an associated polytope, the bounds from \cite[Theorem 2.10]{Sombra99} also lead to improved estimates.

\end{itemize}
\end{remark}

\bigskip

\noindent \textbf{Acknowledgements.} This work was partially supported by the Argentinean research grants
CONICET PIP 0099/11 and UBACYT 20020120100133 (2013-2016).

\end{document}